\documentclass{amsart}
\usepackage[utf8]{inputenc}
\usepackage{amssymb,amsmath,amsthm,latexsym}
\usepackage{enumerate}
\usepackage[margin=1in]{geometry}
\usepackage{xcolor}

\newtheorem{theorem}{Theorem}[section]
\newtheorem{lemma}[theorem]{Lemma}

\theoremstyle{definition}
\newtheorem{prb}[theorem]{Problem}
\newtheorem{rmk}[theorem]{Remark}

\newcommand{\norm}[1]{\left\lVert#1\right\rVert}
\newcommand{\bilin}[1]{\langle#1\rangle}
\newcommand{\RR}{{\mathbb{R}}}

\newcommand{\boldx}{{\boldsymbol x}}
\newcommand{\boldDelta}{{\boldsymbol \Delta}}

\begin{document}

\title[Recovering Gaussian Mixtures from Distance Distributions]{The Shape of a Gaussian Mixture is Characterized by the Probability Density of the Distance between two Samples } 

\author{Mireille Boutin} 
\address{ Mireille~Boutin\\
  Purdue University\\
  Elmore Family School of ECE and  Dept.\ of Mathematics\\
  465 Northwestern Av.\\
  West Lafayette, IN 47907\\ USA}
\email{mboutin@purdue.edu}

\author{Kindyl King}
\address{ Kindyl~King\\
  Purdue University\\
  Dept.\ of Mathematics\\
  150 N.\ University St.\\
  West Lafayette, IN 47907\\ USA}
\email{king487@purdue.edu}

\author{Uli Walther}
\address{ Uli~Walther\\
  Purdue University\\
  Dept.\ of Mathematics\\
  150 N.\ University St.\\
  West Lafayette, IN 47907\\ USA}
\email{walther@purdue.edu}

\thanks{This work was supported in parts by NSF grant DMS-2100288 and by Simons
  Foundation Collaboration Grant for Mathematicians \#580839. }

\date{}

\begin{abstract}
     Let $\boldx$ be a random variable with density $\rho(x)$ taking values in  $\RR^d$. We are interested in finding a representation for the shape of $\rho(x)$, i.e.~for the orbit $\{ \rho(g\cdot x) | g\in E(d) \}$ of $\rho$ under the Euclidean group.
     Let $x_1$ and $x_2$ be two random samples picked, independently, following $\rho(x)$, and let $\Delta$ be the squared Euclidean distance between  $x_1$ and $x_2$.  We show, if $\rho(x)$ is a mixture of Gaussians whose covariance matrix is the identity, and if the means of the Gaussians are in generic position, then the density $\rho(x)$ is reconstructible, up to a rigid motion in $E(d)$, from the density of $\boldDelta$. In other words, any two such Gaussian mixtures $\rho(x)$ and $\bar{\rho} (x)$ with the same distribution of distances are guaranteed to be related by a rigid motion $g\in E(d)$ as $\rho(x)=\bar{\rho} (g\cdot x)$.
We also show that a similar result holds when the distance is defined by a symmetric bilinear form. 
\end{abstract}
\maketitle

\section{Introduction}

It has been shown in previous work \cite{boutin2004reconstructing} that the shape of a generic point configuration is characterized by the distribution of its pairwise distances. More specifically, if two generic collections of $k$ points, say $p_1,p_2,\ldots, p_k \in \RR^d$ and $q_1,q_2,\ldots, q_k \in \RR^d$, have the same multiset of pairwise distances, then there exists a rigid motion $g$ in the Euclidean group $E(d)$, and a permutation $\pi$ in the symmetric group $S_k$,
such that
\[ g\cdot p_{\pi (i)} = q_i, \text{ for all }i=1,\ldots,k. \]
Thus one can compare the shape of two point configurations by comparing their respective multiset of pairwise distances.

Extending this result to the more practical case where the distance measurements are noisy is not straightforward. It is natural to assume that a small difference between two multisets of pairwise distances should correspond to a small difference in the shape of the underlying point configurations. Unfortunately, this is not true, even for \emph{generic} multisets of distances. This is because of there exist point configurations whose shape is not uniquely reconstructible from the multiset of their distances; see  \cite{boutin2004reconstructing} for examples. Indeed, these \emph{exceptional point configurations} make the problem of reconstructing a point configuration from noisy distance measurements ill-conditioned. 

To see why this is the case, take two exceptional point configurations which do have the same multiset of pairwise distances but which are not transformed into one another by any element of the Euclidean group.  Now choose any two generic point configurations in the vicinity of these exceptional configurations, respectively. Those two generic point configurations have similar, but different, pairwise distances. Yet, by construction, their shapes are very different.  Adding a small amount of noise to them will perturb their pairwise distances slightly. Thus the multiset of their pairwise distances will remain similar, even though the shape of their respective underlying  point configurations is still very different. Attempting to reconstruct the point configuration from the noisy distances, one would be hard pressed to figure out which original point configuration is the "correct" one.

This paper explores a new representation for the shape of a point configuration which is observed under noisy conditions. More specifically, the noisy point measurements are viewed as random samples following a Gaussian mixture probability model. The 
means of the Gaussians in the mixture represent the true location of the points in the configuration. The weight of each Gaussian represents the probability that a point will be observed, thus allowing for the possibility of point omission. Finally, the covariance matrix of Gaussian $i$ in the mixture represents the noise in the measurement of point $i$ in the configuration.

Our proposed representation is the distribution of the squared distance between two samples drawn, independently, following the given Gaussian mixture. In the sequel, we show that generic Gaussian mixtures are uniquely determined, up to a rigid motion, by their distribution of squared distances.

\section{Problem Statement}

Let $\boldx$ be a random variable taking values in  $\RR^d$. Denote by $\rho(x):=\rho(\boldx=x)$ be the probability density function of $\boldx$. We are interested in characterizing $\rho(x)$ modulo the action of the Euclidean group $E(d)$ on  $\RR^d$. In other words, given two densities, say $\rho(x)$ and $\bar{\rho} (x)$, we would like to determine if there exists a rigid motion  $g\in E(d)$ such that $\rho(x)=\bar{\rho} (g\cdot x)$.

Let $x_1$ and $x_2$ be two random samples in  $\RR^d$ picked independently, following $\rho(x)$. Let $\Delta$ be the squared Euclidean distance between the two samples
\[
\Delta= \|x_1-x_2 \|^2. \]
Denote by $r(\Delta)$ the probability density function of the corresponding random variable $\boldDelta $,
\[
r(\Delta):=r(\boldDelta=\Delta).
\]
We call $r(\Delta)$ the {\em distribution of distances of $\rho(x)$}. We want to determine under what circumstances $r(\Delta)$ characterizes the shape of $\rho(x)$. 
\begin{prb}
When is it the case that  two densities $\rho(x)$ and $\bar{\rho} (x)$ having the same distribution of distances are guaranteed to be related by a rigid motion $g\in E(d)$ as $\rho(x)=\bar{\rho} (g\cdot x)$?
\end{prb}

\section{The Case of a Mixture of Equally Weighted Gaussians}

\begin{theorem}\label{thm-eq-weight}
Let 
$k \geq d+2$. 
Assume that $\rho(x)$ is a mixture of $k$ equally weighted standard normal distributions: 
\[\rho(x) = \sum_{i=1}^k \frac{1}{k} \rho_i(x) =
\sum_{i=1}^k \frac{1}{ k \sqrt{(2\pi)^d }}\, e^{-\frac{1}{2}\| x-\mu_i\|^2 }.\]
Assume that the point configuration formed by their means $\mu_1,\mu_2,\ldots, \mu_k$ is reconstructible from its multiset of pairwise distances. 

Then  $\rho(x)$  is uniquely reconstructible, up to a rigid motion, from its distribution of distances $r(\Delta)$.
\label{thm:equalweightstandardnormals}
\end{theorem}
\begin{proof}
The moment generating function of $\Delta$ is the expected value
\begin{align*}
    M(t):=E\left(e^{t\Delta}\right) 
        &= \int_{\RR^d} \int_{\RR^d} e^{t\Delta}\rho(x_1)\rho(x_2) dx_1 dx_2,\\
    &= \frac{1}{k^2} \sum_{i,j=1}^k
       \int_{\RR^d} \int_{\RR^d} e^{t\Delta}\rho_i(x_1)\rho_j(x_2) dx_1 dx_2, \\
    &= \frac{1}{k^2} \sum_{i,j=1}^k
       \int_{\RR^d} \int_{\RR^d} e^{t\| x_1-x_2\| ^2} \frac{1}{\sqrt{2\pi}}\, e^{\frac{-\| x_1-\mu_i\| ^2}{2}}                
            \frac{1}{\sqrt{2\pi}}\, e^{\frac{-\| x_2-\mu_j\| ^2}{2}}dx_1 dx_2, \\
    &=: \frac{1}{k^2} \sum_{i,j=1}^k M_{ij}(t).
\end{align*}
The function $M_{ij}(t)$ is the moment generating function of the squared distance $\norm{x_1-x_2}^2$ when $x_1$ is sampled from $\rho_i(x)$ and $x_2$ is sampled from $\rho_j(x)$. Integrating, we find
\[M_{ij}(t) = \frac{1}{\sqrt{1-4t}}\, e^{\frac{t\| \mu_i-\mu_j\| ^2}{1-4t}}.\]
Expanding the exponential, we obtain 
\[M_{ij}(t) = \sum_{n=0}^\infty \frac{\| \mu_i-\mu_j\|^{2n}}{n!} \frac{t^n}{(1-4t)^{n+\frac{1}{2}}}.\]
Therefore, 
\[ M(t)=  \frac{1}{k^2} \sum_{n=0}^\infty \sum_{i,j=1}^k \frac{\| \mu_i-\mu_j\|^{2n}}{n!} \frac{t^n}{(1-4t)^{n+\frac{1}{2}}}. \]
Let us focus on the coefficient $m_n$ of $ \frac{t^n}{(1-4t)^{n+\frac{1}{2}}}$ in the expansion of $M(t)$:
\[
m_n:=\frac{1}{k^2} \sum_{i,j=1}^k \frac{\| \mu_i-\mu_j\|^{2n}}{n!} =  
\frac{2}{k^2 n!} \sum_{1\le i<j\le k} \| \mu_i-\mu_j\|^{2n}.
\]
Knowing  $r(\Delta)$ determines  $M(t)$ and therefore also $m_n$, for every $n=0,1,2,\ldots, \infty$. Thus, $r(\Delta)$ determines the value of the power sums of the squared distances
\[
\ \sum_{i\neq j=1}^k \left( \| \mu_i-\mu_j\|^{2} \right)^n, \text{ for }  n=0,1,2,\ldots, \infty. \]
Knowledge of the first $k$ power sums is sufficient to recover the values of the (unlabeled) squared distances $\| \mu_i-\mu_j\|^2$. Since $k \geq d+2$, and since the point configuration formed by the $\mu_i$'s is generic, Theorem 2.6 in \cite{boutin2004reconstructing} implies that one can recover the coordinates of the (unlabeled) means $\mu_1,\ldots, \mu_k$, up to a rigid motion of the point configuration. The conclusion follows from the fact that the Gaussian mixture is uniquely determined by the values of the $\mu_i$'s.
\end{proof}

\begin{rmk}
Theorem \ref{thm-eq-weight} also holds, without any assumption on the positions of the means of the Gaussian, for $k=2$ and $k=3$ (where all configurations are reconstructible), and in the trivial case $k=1$.

\end{rmk}

\section{The Case of a Mixture of Gaussians}

\begin{theorem}
Let 
$k \geq d+2$. 
Assume that $\rho(x)$ is a mixture of $k$ normal distributions: 
\[\rho(x) = \sum_{i=1}^k  \pi_i \rho_i(x) =
\sum_{i=1}^k \frac{\pi_i}{  \sqrt{(2\pi)^d }}\, e^{-\frac{1}{2}\| x-\mu_i\|^2 },\]
where  $\pi_i >0$ for all $i=1,\dots, k$ and $\sum_{i=1}^k \pi_i = 1$. 
Assume that the means $\mu_1,\mu_2,\ldots, \mu_k$ have pairwise distinct distances $\|\mu_i-\mu_j\|$, and that the configuration of their means can be uniquely reconstructed, up to a rigid motion, from the multiset of their pairwise distances.

Then  $\rho(x)$  is uniquely reconstructible, up to a rigid motion, from its distribution of distances $r(\Delta)$. 
\end{theorem}
\begin{proof}
The moment generating function of $\Delta$ is 
\begin{align*}
    E\left(e^{t\Delta}\right) 
        &= \int_{\RR^d} \int_{\RR^d} e^{t\Delta}\rho(x_1)\rho(x_2) dx_1 dx_2,\\
    &=  \sum_{i,j=1}^k \pi_i \pi_j
       \int_{\RR^d} \int_{\RR^d} e^{t\Delta}\rho_i(x_1)\rho_j(x_2) dx_1 dx_2, \\
    &=  \sum_{i,j=1}^k \pi_i \pi_j
       \int_{\RR^d} \int_{\RR^d} e^{t\| x_1-x_2\| ^2} \frac{1}{\sqrt{2\pi}} e^{\frac{-\| x_1-\mu_i\| ^2}{2}}                
            \frac{1}{\sqrt{2\pi}} e^{\frac{-\| x_2-\mu_j\| ^2}{2}}dx_1 dx_2, \\
    &=:  \sum_{i,j=1}^k \pi_i \pi_j M_{ij}(t).
\end{align*}Again, the function $M_{ij}(t)$ is the moment generating function of the squared distance $\norm{x_1-x_2}^2$ when $x_1$ is sampled from $\rho_i(x)$ and $x_2$ is sampled from $\rho_j(x)$. 
From the proof of Theorem \ref{thm:equalweightstandardnormals}, we have
\[M_{ij}(t) = \sum_{n=0}^\infty \frac{\| \mu_i-\mu_j\|^{2n}}{n!} \frac{t^n}{(1-4t)^{n+\frac{1}{2}}}.\]
Therefore, 
\[ M(t)=  \sum_{n=0}^\infty \sum_{i,j=1}^k \pi_i \pi_j \frac{\| \mu_i-\mu_j\|^{2n}}{n!} \frac{t^n}{(1-4t)^{n+\frac{1}{2}}}. \]
Let us focus on the coefficient $m_{\pi,n}$ of $ \frac{t^n}{(1-4t)^{n+\frac{1}{2}}}$ in the expansion of $M(t)$:
\[ m_{\pi,n}:=\sum_{i,j=1}^k \pi_i \pi_j \frac{\| \mu_i-\mu_j\|^{2n}}{n!} = 
\frac{2}{ n!} \sum_{1\le i<j\le k} \pi_i \pi_j \| \mu_i-\mu_j\|^{2n}
\]
Since  $r(\Delta)$ determines $M(t)$, is also determines $m_{\pi,n}$, for every $n=0,\ldots, \infty$. Thus it also determines the value of the weighted power sums of the squared distances
\[ \sum_{1\le i<j\le k} \pi_i \pi_j \left( \| \mu_i-\mu_j\|^{2} \right)^n, \text{ for }  n=0,1,2,\ldots, \infty. \]
By Lemma \ref{lem:weightedpowersums} below, using the hypothesis that the point configuration formed by the $\mu_i$'s does not contain repeated distances, one can recover the values of the (unlabeled) pairs $\{ (\pi_i \pi_j,\| \mu_i-\mu_j\|^2)\}_{i\neq j, i,j=1}^k$.
Since the $\pi_i$'s are non-zero, we can use the values of the products $\pi_i \pi_j$ to recover the values of the $\pi_i$'s up to a common sign factor, while keeping track of which $\pi_i$ and $\pi_j$ is associated to which $ \| \mu_i-\mu_j\|^2$. As every $\pi_i$ must be positive, one may in fact recover the $\pi_i$ from the multiset of their pairwise products.

According to \cite{boutin2004reconstructing}, assuming that the point configuration formed by the $\mu_i$'s is generic, one can recover the coordinates of the (unlabeled) means $\mu_1,\ldots, \mu_k$, up to a rigid motion of the point configuration. Picking an ordering for (a specific incarnation modulo rigid motions of) the means, 
The conclusion follows from the fact that a Gaussian mixture is uniquely determined by the weights $\pi_i$ and the mean $\mu_i$ for each Gaussian component.
\end{proof}
Note that the  result also holds, without any assumption on the positions of the means of the Gaussian, for the trivial case of $k=1$, as well as for $k=2$ and $k=3$.

\begin{lemma}
Let $a_1,a_2,\ldots, a_k \in \RR$ be non-zero real numbers.  Let $x_1,x_2, \ldots, x_k \in \RR$ be pairwise distinct real numbers. Then given the values of the weighted power sums
\begin{eqnarray}
\label{eqn-pn}
 p_n&=&\sum_{i=1}^k a_i x_i^n, \text { for } n=0,1,2,\ldots, 2k-1,
 \end{eqnarray}
one can recover the values of the (unlabeled) pairs $(a_i,x_i)$. 
\label{lem:weightedpowersums}
\end{lemma}

\begin{proof}
Consider the polynomial \[ F(t)=\prod_{i=1}^k (t-x_i)= t^k- \sum_{r=1}^k c_r t^{k-r},\] where the $c_r$ is, up to a sign, the $r$-th elementary symmetric polynomials in the $x_i$'s. 

We have $F(x_i)=0$, for any $i=1,2,\ldots,k$. Thus we have also have $ a_ix_i^m  F(x_i)=0$, for any $i=1,2,\ldots,k$ and any $m=0,1,2, \ldots, \infty $. Summing over $i$, we get
\begin{eqnarray*}
0&=&\sum_{i=1}^k a_i x_i^m F(x_i), \\
&=&\sum_{i=1}^k a_i x_i^m \left( x_i^k- \sum_{r=1}^k c_r x_i ^{k-r} \right)
\end{eqnarray*}, and so 
\[0=p_{k+m} - \sum_{r=1}^k c_r p_{k-r+m}.
\]
Writing these equations in matrix form, for $m=0,1,2,\ldots, k-1$, we have
\[
\left( 
\begin{array}{c}
p_k\\
p_{k+1} \\
p_{k+2} \\
\vdots \\
p_{2k-1}
\end{array}
\right)
=
\left( 
\begin{array}{ccccc}
p_0& p_1 & p_2 & \ldots & p_{k-1} \\
p_{1} & p_2 & p_3 & \ldots & p_{k} \\
p_{2} & p_3 & p_4 & \ldots & p_{k+1} \\
\vdots &&&&\vdots\\
p_{k-1} & p_{k} & p_{k+1} & \ldots & p_{2k-1}
\end{array}
\right)
\left( 
\begin{array}{c}
c_k\\
c_{k-1} \\
c_{k-2} \\
\vdots \\
c_{1}
\end{array}
\right).
\]
The matrix on the right-hand-side can be factored as
\[ 
\left( 
\begin{array}{ccc}
p_0&  \ldots & p_{k-1} \\
p_{1}& \ldots & p_{k} \\
p_{2} & \ldots & p_{k+1} \\
\vdots&&\vdots \\
p_{k-1} & \ldots & p_{2k-1}
\end{array}
\right) = \underbrace{\left( 
\begin{array}{cccc}
1 & 1  & \ldots & 1 \\
x_1 & x_2  & \ldots & x_k  \\
x_1^2 & x_2^2  & \ldots & x_k^2  \\
\vdots&&&\vdots \\
x_1^{k-1} & x_2^{k-1}  & \ldots & x_k^{k-1}  \\
\end{array}
\right)}_{:=\Xi} 
\underbrace{\left( 
\begin{array}{cccc}
a_1& 0   & \ldots & 0\\
0 & a_2   & \ldots & 0 \\
\vdots &&&\vdots\\
0 & 0 &  \ldots & a_k
\end{array}
\right)}_{:=A}
\underbrace{\left( 
\begin{array}{ccccc}
1& x_1& x_1^2  & \ldots & x_1^{k-1} \\
1 & x_2 & x_2^2 & \ldots & x_2^{k-1}\\
1 & x_3 & x_3^2 & \ldots & x_3^{k-1}\\
\vdots &&&&\vdots\\
1 & x_k & x_k^2 & \ldots & x_k^{k-1}
\end{array}
\right)}_{=\Xi^T}.
\]
The matrix $A$ is invertible since the $a_i$'s a non-zero. The matrices $\Xi, \Xi^T$ are Vandermonde matrices, with determinant $\prod_{1\le i<j\le k}(x_j-x_i)$ and thus invertible when the $x_i$'s are pairwise distinct. Thus we can recover the value of the coefficients $c_r$'s from $p_0,p_1,p_2,\ldots, p_{2k-1}$. The values of the $c_r$'s uniquely prescribe the set of roots of the polynomial $F(t)$, and thus the $x_i$'s are determined as a set. Label them (arbitrarily) and then solve the following system of equation in order to recover the corresponding $a_i$'s.
\[
\left( 
\begin{array}{c}
p_0\\
p_{1} \\
\vdots \\
p_{k-1}
\end{array}
\right)=
\Xi\cdot 
\left( 
\begin{array}{c}
a_1\\
a_2 \\
\vdots \\
a_k
\end{array}
\right),
\]
compare Equation \eqref{eqn-pn}.
 Again, this is possible because  $\Xi$ is invertible when the $x_i$'s are distinct.
\end{proof}

More generally, we have
\begin{theorem}
Let $V=\RR^d$ be equipped with a non-degenerate symmetric bilinear form $\bilin{-,-}$. Let $\boldx $ be a random variable taking values in $V$ following a probability distribution 
\[ \rho(x)=  \sum_{i=1}^k  \pi_i \rho_i(x) =
\sum_{i=1}^k \frac{\pi_i}{  \sqrt{(2\pi)^d }}\, e^{-\frac{1}{2} \bilin{ x-\mu_i,x-\mu_i} },\]
where  $\pi_i >0$ for all $i=1,\dots, k$ and $\sum_{i=1}^k \pi_i = 1$.
Pick two samples $x_1$, $x_2$, independently, following $\rho(x)$. Let $\Delta = \bilin{x_1-x_2, x_1-x_2}$. Denote by $r(\Delta)$ the probability distribution of the random variable ${\bf \Delta}$ thus defined.

If 
$k \geq d+2$, and if  the means $\mu_1,\mu_2,\ldots, \mu_k$ are generic, 
then $\rho(x)$ is uniquely reconstructible from $r(\Delta)$, up to a transformation $g\in O(V)$, the symmetry group of the bilinear form  $\bilin{-,-}$.
\end{theorem}
\begin{proof}
The proof is the same as for the previous theorem.
\end{proof}

Again, the  result also holds, without any assumption on the positions of the means of the Gaussian, for the trivial case of $k=1$, as well as for $k=2$ and $k=3$.

\bibliographystyle{unsrt}
\bibliography{bibio}
\end{document}